\documentclass{amsart}
\usepackage{fontenc}
\usepackage{lmodern}
\usepackage{amsmath,amssymb,mathtools}
\usepackage{microtype,xurl,enumitem,needspace,etoolbox}
\usepackage[hidelinks]{hyperref}

\newtheorem{theorem}{Theorem}[section]
\newtheorem{proposition}[theorem]{Proposition}

\newtheorem{corollary}[theorem]{Corollary}
\theoremstyle{definition}

\newtheorem*{question}{Question}
\theoremstyle{remark}

\newcommand{\Z}{\mathbb Z}
\newcommand{\Q}{\mathbb Q}

\newcommand{\Zp}{\mathbb Z_p}
\newcommand{\Qp}{\mathbb Q_p}

\newcommand{\rhobar}{\overline\rho}
\newcommand{\bV}{\overline V}
\newcommand{\bS}{\overline S}
\newcommand{\bQ}{\overline Q}
\newcommand{\bB}{\overline B}
\newcommand{\eps}{\varepsilon}
\newcommand{\beps}{\overline\varepsilon}
\DeclareMathOperator{\GL}{GL}
\DeclareMathOperator{\Gal}{Gal}
\DeclareMathOperator{\Hom}{Hom}
\DeclareMathOperator{\End}{End}
\DeclareMathOperator{\Soc}{Soc}
\DeclareMathOperator{\tr}{tr}
\DeclareMathOperator{\Tr}{Tr}

\DeclareMathOperator{\charac}{char}
\numberwithin{equation}{section}

\title[Lifting mod $p$ Galois representations over local fields]
{Lifting mod $p$ Galois representations over local fields}
\author{Yufan Luo}
\address{Shanghai Institute for Mathematics and Interdisciplinary Sciences (SIMIS), Shanghai 200433, China}
\address{Research Institute of Intelligent Complex Systems, Fudan University, Shanghai 200433, China}


\begin{document}

\begin{abstract}
Let $K$ be a local field with absolute Galois group $G_{K}$ and let $k$ be a finite field of characteristic $p$. We prove that every continuous
representation $G_K\to\GL_n(k)$ lifts to a continuous representation
$G_K\to\GL_n(W(k))$ where $W(k)$ is the ring of infinite Witt vectors over $k$.
\end{abstract}

\maketitle
\tableofcontents

\section{Introduction}\label{sec:introduction}

Let $p$ be a prime and let $k$ be a finite field of characteristic $p$. Let $W(k)$ be the ring of infinite Witt vectors over $k$. Recall that a local field is a locally compact, non-discrete Hausdorff topological field. The problem of lifting residual Galois representations has a long history;
see, for instance, \cite{Boc03}. We consider the following question.

\begin{question}
Let $K$ be a local field with absolute Galois group $G_K$, and let
$n$ be a positive integer. Does every continuous representation $  \rhobar:G_K\longrightarrow\GL_n(k)$ admit a continuous lift to $W(k)$, i.e., a continuous representation $\widetilde{\rho}:G_K\longrightarrow\GL_n(W(k))$
whose reduction modulo $p$ is $\rhobar$?
\end{question}

The main theorem of this paper gives an affirmative answer to the above question.

\begin{theorem}\label{thm:main}
Let $K$ be a local field, let $p$ be any prime, and let $k$ be a finite
field of characteristic $p$. Then, for every integer $n\geq1$ and every
continuous representation $\rhobar:G_K\to\GL_n(k)$, there exists a
continuous representation $\rho:G_K\longrightarrow\GL_n(W(k))$ whose reduction modulo $p$ is $\rhobar$.
\end{theorem}

\subsection{Relation to previous work}

The local lifting problem has been studied both over a prescribed
coefficient ring and after extending the coefficients. To distinguish
these questions, we put $E_0=W(k)[1/p]$. A characteristic-zero lift of $\rhobar$ means a finite extension $E/E_0$
and a continuous representation
\[
 \rho_E:G_K\longrightarrow\GL_n(\mathcal O_E),\qquad
 \rho_E\bmod\mathfrak m_E\simeq\rhobar\otimes_k k_E,
\]
where $\mathcal O_E$, $\mathfrak m_E$, and $k_E$ denote the ring of integers,
maximal ideal, and residue field of $E$. This formulation allows
$E/E_0$ to be ramified and $k_E$ to be larger than $k$.

B\"ockle proved that every continuous residual representation of the
absolute Galois group of a local field lifts to
$W_2(k)=W(k)/p^2W(k)$ in \cite[Theorem~1.1]{Boc03}. Furthermore, B\"ockle proved the full $W(k)$-lifting statement when the local field $K$ has characteristic $p$ in \cite[Theorem~1.7]{Boc03}. The archimedean cases are also already treated by B\"ockle in \cite[Lemma~2.3]{Boc03}.
He also proved that a lift over $W(k)$ exists whenever the finite
Galois extension of $K$ cut out by $\rhobar$ contains no primitive
$p$-th root of unity; see \cite[Proposition~1.4]{Boc03}.

For a nonarchimedean local field of residue characteristic
$\ell\ne p$, B\"ockle gave a $W(k)$-lift whenever
the residual inertia image is a $p$-group; see
\cite[Theorem~1.3]{Boc03}. When, in addition, $\charac K=0$,
characteristic-zero liftability for arbitrary residual representations
was established by Choi in \cite[Theorem~1.1]{Cho17}.
See also the minimally ramified lifting results of Clozel, Harris,
and Taylor \cite[Corollary~2.4.21]{CHT08}. For any nonarchimedean local field of residue characteristic
$\ell\ne p$, Booher, Cotner, and Tang's results give a lift over
$W(k)$ for every $\rhobar$ when $p>n$; see
\cite[Theorem~1.1 and Remark~6.21]{BCT24}, applied to
$\GL_n$ over $W(k)$.

For finite extensions $K$ of the field $\Q_p$ of $p$-adic numbers, Emerton and Gee proved that every
continuous representation $\rhobar:G_K\to\GL_n(k)$ admits a
characteristic-zero lift over $\mathcal O_E$ for some finite extension
$E/E_0$; see \cite[Theorem~6.4.4]{EG}. B\"ockle, Iyengar, and Pa\v{s}k\={u}nas gave another proof of
characteristic-zero liftability after a finite extension of coefficients
in the $p$-adic case using local Galois deformation rings;
see \cite[Corollary~3.45]{BIP}.

A different approach retaining $W(k)$-coefficients is due to
Conti, Demarche, and Florence. For a residual representation
admitting a unique complete invariant flag, they construct
compatible lifts to all $W_r(k)=W(k)/p^r W(k)$; see
\cite[Corollary~7.25(1) and Remark~7.28]{CDF24}.

In this paper, we \emph{completely} resolve the lifting problem for local fields by constructing lifts \emph{over the original Witt ring $W(k)$}. Specifically, we establish $W(k)$-liftability for all nonarchimedean local fields $K$ of $\charac K \ne p$, for arbitrary $p$, $n$, and $\bar{\rho}$, without extending $E_0$ or imposing any flag condition. Furthermore, our approach treats all residue characteristics uniformly and is independent of previous lifting results.

\subsection{Sketch of the proof}
As noted above, it suffices to consider nonarchimedean local fields $K$ with $\charac K\ne p$. Our first observation is that every finite quotient of $G_K$ is solvable. By a lifting theorem for irreducible mod $p$ representations of finite
solvable groups due to Isaacs and Schmid \cite{Isa74,Sch83}, every
irreducible continuous representation $\rhobar:G_K\to\GL_n(k)$ admits
a lift $M$ over $W(k)$ for which
\begin{equation}\label{eq:intro-end}
 \End_{G_K}(M)\longrightarrow\End_{G_K}(\rhobar)
 \quad\text{is surjective}.
\end{equation}
The key to the proof is the following criterion from a generalized Tate local duality for Galois cohomology: if $A$ and $B$ are finite free $W(k)$-modules with continuous $W(k)$-linear
$G_K$-actions for the $p$-adic topology and
\begin{equation}\label{eq:intro-hom}
 \Hom_{G_K}(A,B(1))\longrightarrow
 \Hom_{G_K}(A/pA,(B/pB)(1))
 \quad\text{is surjective},
\end{equation}
then every extension of $B/pB$ by $A/pA$ with continuous
$G_K$-action is obtained by reducing an exact sequence
\[
 0\longrightarrow A\longrightarrow V\longrightarrow B
 \longrightarrow0
\]
of finite free $W(k)$-modules with continuous $W(k)$-linear
$G_K$-actions and $G_K$-equivariant maps.

To apply this criterion, we proceed by induction on the dimension of the given
residual representation $\bV$ of $G_{K}$. The \emph{socle} $\Soc(\bV)$ of $\bV$ is
the sum of all irreducible $G_K$-subrepresentations
of $\bV$. We carry out the induction with the additional requirement that
$V$ contain a $G_K$-stable submodule $S$ lifting the inclusion
$\Soc(\bV)\hookrightarrow\bV$. Here we require $S$ to be obtained
by choosing a lift satisfying \eqref{eq:intro-end} for each
irreducible constituent of $\Soc(\bV)$ and retaining its multiplicity.
The semisimple case follows by taking direct sums of such lifts and hence it suffices to consider the non-semisimple case.

For the induction step, we put
\[
 \bS=\Soc(\bV),\qquad \bQ=\bV/\bS,\qquad
 \bB=\Soc(\bQ).
\]
Since $\dim_k\bQ<\dim_k\bV$, induction provides a lift $Q$
containing a submodule $B$ that lifts the inclusion
$\bB\hookrightarrow\bQ$, with $B$ a direct sum of lifts as above.
We choose the lift $S$ of $\bS$ to match these choices: if an
irreducible constituent $\sigma$ of $\bS$ is isomorphic to
$\tau(1)$ for a constituent $\tau$ of $\bB$, and $M_\tau$ is
the chosen lift of $\tau$ in $B$, we lift $\sigma$ by $M_\tau(1)$.
For the remaining constituents, we choose lifts satisfying
\eqref{eq:intro-end}.

Since $\bS$ is semisimple, every $G_K$-equivariant map $f:\bS\to\bQ(1)$ has image in $\Soc(\bQ(1))=\bB(1)$.
Write this map in blocks relative to the irreducible decompositions
of $\bS$ and $\bB(1)$. Blocks between nonisomorphic irreducible
types are zero. Each remaining block is a matrix of endomorphisms
of a single irreducible representation, and our choices use the
same $W(k)$-lift for its source and target summands.
The surjectivity in \eqref{eq:intro-end}, which is preserved by
Tate twists, therefore allows us to lift each matrix entry.
Assembling these lifts gives a map $S\to B(1)\hookrightarrow Q(1)$
reducing to $f$. Thus \eqref{eq:intro-hom} holds with
$(A,B)=(S,Q)$, and the criterion lifts
$0\to\bS\to\bV\to\bQ\to0$ to
$0\to S\to V\to Q\to0$, completing the induction.

\subsection{Notation and conventions}

Throughout this paper, we fix a finite field $k$ of characteristic $p$. Let $\Q_p$ be the field of $p$-adic numbers and $\Zp$ be the ring of integers of $\Q_p$. For any field $K$, we write $\charac K $ for the characteristic of $K$ and $G_K=\Gal(K^{\mathrm s}/K)$ for the absolute Galois group of $K$, where $K^{\mathrm s}$ is a separable
closure of $K$. For a group $G$ and $R=k$ or $W(k)$, a representation
$\rho:G\to\GL_n(R)$ equips $R^n$ with a left $R[G]$-module
structure where $R[G]$ is the group ring. We identify a representation with this module when convenient. Finite-dimensional $k$-vector spaces
are equipped with the discrete topology, and finite free $W(k)$-modules
with the $p$-adic topology.
 
 If $K$ is a nonarchimedean local field with $\charac K\ne p$, then we let
$\eps:G_K\to\Zp^\times\subset W(k)^\times$ denote the $p$-adic
cyclotomic character of $G_{K}$, with reduction $\beps$. For a finite free
$W(k)$-module $M$ with a continuous $W(k)$-linear $G_K$-action,
a $k[G_K]$-module $\rhobar$, and $r\in\Z$, the \emph{Tate twists}
$M(r)$ and $\rhobar(r)$ have the same underlying modules, with
$g\in G_K$ acting by $m\mapsto\eps(g)^r gm$ and
$v\mapsto\beps(g)^r gv$, respectively.

\section{Lifting irreducible mod $p$ representations}\label{sec:simple}

Our first input is the following consequence of a lifting theorem
for irreducible mod $p$ representations of finite solvable groups
due to Isaacs and Schmid  \cite{Isa74,Sch83}.

\begin{proposition}\label{prop:finite-simple}
Let $H$ be a finite solvable group, and let $\rhobar$ be an
irreducible $k[H]$-module which is finite-dimensional over $k$. Then there is a finite free $W(k)$-module $M$ lifting $\rhobar$, equipped with a $W(k)$-linear $H$-action, such that the natural reduction map
\[
 \End_{W(k)[H]}(M)\longrightarrow\End_{k[H]}(\rhobar)
\]
is surjective.
\end{proposition}

\begin{proof}
Note that $E:=\End_{k[H]}(\rhobar)$ is a finite field by
Schur's lemma and Wedderburn's theorem.
Let $F=W(k)[1/p]$. By \cite[Theorem, p.~461]{Sch83},
there are an $F[H]$-module $V$ and an $H$-stable
$W(k)$-lattice $M\subset V$ with $M/pM\simeq\rhobar$
such that $\End_{F[H]}(V)$ is an unramified extension
of $F$ with residue field $E$.
By \cite[Lemma~2]{Sch83},
$A:=\End_{W(k)[H]}(M)$ is the ring of integers of this extension.

Since the natural reduction map $A\to E$ has kernel $pA$,
it induces an injection $A/pA\hookrightarrow E$.
Since $\End_{F[H]}(V)$ is unramified over $F$ with residue
field $E$, the uniformizer $p$ of $W(k)$ remains a uniformizer
of $A$. Hence $pA$ is the maximal ideal of $A$ and
$\dim_k(A/pA)=[E:k]=\dim_k E$.
This injection is therefore an isomorphism, and the natural
reduction map $A\to E$ is surjective.
\end{proof}

As a corollary, we obtain the following result.
\begin{corollary}\label{cor:local-simple}
Let $K$ be a nonarchimedean local field, and let $\rhobar:G_{K}\to \GL(\bV)$ be an irreducible continuous representation of $G_{K}$ on the finite-dimensional $k$-vector space $\bV$. Then $\rhobar$ admits a lift to a finite free $W(k)$-module $M$ with a continuous $W(k)$-linear
$G_K$-action such that the natural reduction map
\[
 \End_{G_K}(M)\longrightarrow\End_{G_K}(\rhobar)
\]
is surjective. When $\charac K\ne p$, $M(1)$ lifts $\rhobar(1)$,
and the natural reduction map
\[
 \End_{G_K}(M(1))\longrightarrow\End_{G_K}(\rhobar(1))
\]
is also surjective.
\end{corollary}

\begin{proof}
Let $H$ be the image of $\rhobar$. Then it is the Galois group of a finite Galois extension of $K$. If $P\subset I\subset H$ are its wild
inertia and inertia groups and $\ell$ is the residue characteristic,
then $P$ is an $\ell$-group, $I/P$ is cyclic, and $H/I$ is cyclic; see \cite[Chapter~IV, Section~2]{SerreLF}. Thus $H$ is solvable.
Apply Proposition~\ref{prop:finite-simple} to $H$ and inflate the
lift $M$ to $G_K$. The surjection $G_K\to H$ gives
$\End_{G_K}(M)=\End_{W(k)[H]}(M)$ and
$\End_{G_K}(\rhobar)=\End_{k[H]}(\rhobar)$.
The reduction map is unchanged and hence remains surjective.
When $\charac K\ne p$, it follows from $\eps\bmod p=\beps$ that
$M(1)/pM(1)\simeq\rhobar(1)$ as $k[G_K]$-modules.
Since twisting multiplies each group operator by an invertible scalar,
$\End_{G_K}(M(1))=\End_{G_K}(M)$ and
$\End_{G_K}(\rhobar(1))=\End_{G_K}(\rhobar)$. Since both reduction maps send $f$ to $f\bmod p$, the natural reduction map $\End_{G_K}(M(1))\longrightarrow\End_{G_K}(\rhobar(1))$ is also surjective.
\end{proof}

\section{A generalized Tate local duality and an extension criterion}\label{sec:duality}

In this section, $K$ is a nonarchimedean local field with
$\charac K\ne p$, and we put $\mathcal O=W(k)$ and
$F=\mathcal O[1/p]$.

For a profinite group $G$ and a topological $\mathcal O$-module $M$
with a continuous $\mathcal O$-linear $G$-action, let
$C^q(G,M)$ be the $\mathcal O$-module of continuous maps
$G^q\to M$. The usual group cohomology differential makes
$C^\bullet(G,M)$ a complex and its cohomology is the
\emph{continuous cohomology} of $G$ with coefficients in $M$, denoted by
\[
 H^q(G,M)=H^q(C^\bullet(G,M)).
\]
When $M$ is discrete, this agrees with the usual continuous group
cohomology. In degree one,
$H^1(G,M)=Z^1(G,M)/B^1(G,M)$, where $Z^1(G,M)$ consists of
continuous maps $c:G\to M$ satisfying
\[
 c(gh)=c(g)+g c(h)\qquad(g,h\in G),
\]
and $B^1(G,M)$ consists of the maps $g\mapsto gm-m$ with
$m\in M$. 

Now we consider the case of $G=G_{K}$. For a finite free $\mathcal O$-module $T$ with a continuous
$\mathcal O$-linear $G_K$-action for the $p$-adic topology, there are
canonical isomorphisms
\begin{equation}\label{eq:continuous-cohomology-limit}
 H^q(G_K,T)\simeq\varprojlim_r H^q(G_K,T/p^rT)
 \qquad(q\geq0),
\end{equation}
by \cite[Lemma~4.2.2, pp.~96--97]{Nek06}, since the cohomology
groups of finite discrete $p$-primary $G_K$-modules are finite by
\cite[Chapter~I, Theorem~2.1]{Mil06}.
We give $F/\mathcal O$ the discrete topology and trivial
$G_K$-action, and put
\[
 D_T=\Hom_{\mathrm{cts},\mathcal O}(T,F/\mathcal O)(1).
\]
We give $D_T$ the discrete topology and the action
\[
 (g\cdot f)(t)=\eps(g)f(g^{-1}t).
\]
Every continuous homomorphism $T\to F/\mathcal O$ factors through
$T/p^rT$ for some $r$. Thus $D_T$ is the union of the finite
modules $\Hom_{\mathcal O}(T/p^rT,F/\mathcal O)(1)$, each with a
continuous $G_K$-action, and its discrete $G_K$-action is continuous.
The following result is a generalization of Tate local duality
from finite coefficients to $\mathcal O$ coefficients.

\begin{proposition}[Tate--Nekov\'a\v{r}]\label{prop:integral-local-duality}
Let $T$ be a finite free $\mathcal O$-module with a continuous
$\mathcal O$-linear $G_K$-action for the $p$-adic topology.
Then there is an isomorphism of $\mathcal O$-modules
\[
 H^2(G_K,T)\xrightarrow{\sim}
 \Hom_{\mathcal O}(D_T^{G_K},F/\mathcal O).
\]
\end{proposition}

\begin{proof}
This is the case $R=\mathcal O$ and $i=2$ of
\cite[Section 5.2.10, p.~121; see also Theorem~5.2.6, p.~119]{Nek06}.
For completeness, we give a proof using the usual finite-coefficient local Tate duality.

For $r\geq1$, set
\[
 T_r=T/p^rT,\qquad
 D_r=\Hom_{\mathcal O}(T_r,F/\mathcal O)(1).
\]
The group $H^2(G_K,T_r)$ is finite by
\cite[Chapter~I, Theorem~2.1]{Mil06}.
Let $\tau:F/\mathcal O\to\Qp/\Zp$ be induced by $\Tr_{F/\Qp}$.
Since $F/\Qp$ is unramified, it induces an isomorphism
\[
 F/\mathcal O\xrightarrow{\sim}
 \Hom_{\Zp}(\mathcal O,\Qp/\Zp),
 \qquad y\longmapsto\bigl[a\longmapsto\tau(ay)\bigr].
\]
Thus, for every finite $\mathcal O$-module $X$, we have an isomorphism
\[
 \Hom_{\mathcal O}(X,F/\mathcal O)
 \xrightarrow{\sim}\Hom_{\Zp}(X,\Qp/\Zp),
 \qquad f\longmapsto\tau\circ f.
\]

Taking $X=T_r$ and applying the Tate twist yields the natural
$G_K$-equivariant isomorphism
\begin{equation}\label{eq:trace-tate-dual}
 D_r\xrightarrow{\sim}\Hom_{\Zp}(T_r,\Qp/\Zp)(1).
\end{equation}

Applying \cite[Chapter~I, Corollary~2.3]{Mil06} to $T_r$ and
$\Hom_{\Zp}(T_r,\Qp/\Zp)(1)$, and using
\eqref{eq:trace-tate-dual}, we obtain
\begin{equation}\label{eq:finite-tate-original}
 D_r^{G_K}\xrightarrow{\sim}
 \Hom_{\Zp}(H^2(G_K,T_r),\Qp/\Zp).
\end{equation}

Taking $X=D_r^{G_K}$ in the same construction gives an isomorphism
\begin{equation}\label{eq:trace-fixed-dual}
 \Hom_{\mathcal O}(D_r^{G_K},F/\mathcal O)
 \xrightarrow{\sim}\Hom_{\Zp}(D_r^{G_K},\Qp/\Zp),
 \qquad f\longmapsto\tau\circ f.
\end{equation}
Consequently, we have the following chain of isomorphisms:
\begin{equation}\label{eq:finite-integral-duality}
 \begin{aligned}
 H^2(G_K,T_r)
 &\simeq\Hom_{\Zp}\bigl(
   \Hom_{\Zp}(H^2(G_K,T_r),\Qp/\Zp),\Qp/\Zp\bigr)\\
 &\simeq\Hom_{\Zp}(D_r^{G_K},\Qp/\Zp)\\
 &\simeq\Hom_{\mathcal O}(D_r^{G_K},F/\mathcal O).
 \end{aligned}
\end{equation}
In \eqref{eq:finite-integral-duality}, the first isomorphism is
canonical biduality, given by $c\mapsto(\lambda\mapsto\lambda(c))$.
The second is obtained by applying $\Hom_{\Zp}(-,\Qp/\Zp)$ to
\eqref{eq:finite-tate-original}. The third is the inverse of
\eqref{eq:trace-fixed-dual}.

The quotient map $T_{r+1}\to T_r$ is dual to the inclusion
$D_r\hookrightarrow D_{r+1}$. By naturality of local Tate
duality, \eqref{eq:finite-integral-duality} is compatible with these
transition maps.
Since $D_T=\bigcup_rD_r$, we have
\begin{equation}\label{eq:dual-fixed-union}
 D_T^{G_K}=\bigcup_rD_r^{G_K}.
\end{equation}
We can therefore pass to the inverse limit:
\begin{equation}\label{eq:integral-duality-limit}
 \begin{aligned}
 H^2(G_K,T)
 &\simeq\varprojlim_r H^2(G_K,T_r)\\
 &\simeq\varprojlim_r\Hom_{\mathcal O}(D_r^{G_K},F/\mathcal O)\\
 &\simeq\Hom_{\mathcal O}(D_T^{G_K},F/\mathcal O).
 \end{aligned}
\end{equation}
In \eqref{eq:integral-duality-limit}, the first isomorphism is
\eqref{eq:continuous-cohomology-limit} with $q=2$. The second is
obtained by taking inverse limits in \eqref{eq:finite-integral-duality}.
The third follows from \eqref{eq:dual-fixed-union}, since an
$\mathcal O$-linear map from a union of an increasing sequence of
submodules is equivalent to a compatible family of maps from those
submodules.
\end{proof}

We now deduce the extension criterion.

\begin{proposition}\label{prop:extension}
Let $A$ and $B$ be finite free $\mathcal O$-modules with continuous
$\mathcal O$-linear $G_K$-actions. Suppose that the natural reduction map
\begin{equation}\label{eq:hom-surj}
 \Hom_{G_K}(A,B(1))\longrightarrow
 \Hom_{G_K}(A/pA,(B/pB)(1))
\end{equation}
is surjective.
Then every exact sequence of continuous $k[G_K]$-modules 
\[
 0\longrightarrow A/pA\longrightarrow\bV
 \longrightarrow B/pB\longrightarrow0
\]
is the reduction of an exact sequence of $\mathcal O[G_K]$-modules
\[
 0\longrightarrow A\longrightarrow V\longrightarrow B
 \longrightarrow0,
\]
where $V$ is a finite free $\mathcal O$-module with a continuous $\mathcal O$-linear $G_K$-action.
\end{proposition}

\begin{proof}
Put $T=\Hom_{\mathcal O}(B,A)$,
with the $p$-adic topology and $G_K$-action $(g\cdot f)(b)=g(f(g^{-1}b))$ for any $g\in G_K, f\in \Hom_{\mathcal O}(B,A),b\in B $. Since $B$ is free over $\mathcal O$, extensions
\[
 0\longrightarrow A\longrightarrow V\longrightarrow B
 \longrightarrow0
\]
with continuous $\mathcal O$-linear $G_K$-actions are classified, up to
isomorphisms inducing the identity on $A$ and $B$, by $H^1(G_K,T)$.
Indeed, choose an $\mathcal O$-linear splitting $s:B\to V$ and
identify $V$ with $A\oplus B$ via $(a,b)\mapsto a+s(b)$.
In these coordinates, the action is
\[
 g(a,b)=\bigl(ga+c(g)(gb),gb\bigr),
 \qquad c(g)(b)=g s(g^{-1}b)-s(b),
\]
where $c:G_K\to T$ is a continuous $1$-cocycle. Conversely,
every continuous $1$-cocycle defines an extension by this formula.
For two such cocycles $c$ and $c'$, every $\mathcal O$-linear
isomorphism between the corresponding extensions inducing the
identity on $A$ and $B$ has the form
\[
 \Phi_u(a,b)=(a-u(b),b),\qquad u\in T.
\]
It is $G_K$-equivariant if and only if
\[
 c'(g)-c(g)=g\cdot u-u\qquad\text{for all }g\in G_K.
\]
Thus the extension classes are in bijection with
$Z^1(G_K,T)/B^1(G_K,T)=H^1(G_K,T)$.
Similarly, the continuous residual extensions
are classified by $H^1(G_K,T/pT)$, using the natural identification
$T/pT\simeq\Hom_k(B/pB,A/pA)$. Under these identifications,
reduction of extensions corresponds to the natural map
\[
 H^1(G_K,T)\longrightarrow H^1(G_K,T/pT).
\]
It therefore suffices to prove that this map is surjective.
Since the cohomology sequence of
$0\to T\xrightarrow{p}T\to T/pT\to0$ gives an exact sequence
\[
 H^1(G_K,T)\longrightarrow H^1(G_K,T/pT)
 \longrightarrow H^2(G_K,T)\xrightarrow{\ p\ }H^2(G_K,T),
\]
it is enough to show that $H^2(G_K,T)[p]=0$, where $H^2(G_K,T)[p]=\{x\in H^2(G_K,T)\mid px=0\}$ denotes the subgroup annihilated by $p$.

Set
\[
 N=\Hom_{\mathcal O}(A,B(1)),\qquad D=N[1/p]/N,
\]
where $D$ has the discrete topology. Using the common underlying
$\mathcal O$-module of $B(1)$ and $B$, we have an isomorphism 
\begin{align}\label{firstiso}
 \Phi:N&\xrightarrow{\sim}\Hom_{\mathcal O}(T,\mathcal O)(1),\\
n&\longmapsto\bigl[t\longmapsto\tr(t\circ n)\bigr],\notag
\end{align}

where the trace is taken on $A$. In bases of $A$ and $B$,
$\tr(t\circ n)=\sum_{i,j}t_{ij}n_{ji}$, and hence the coefficients of
an $\mathcal O$-linear functional on $T$ determine a unique $n$.
Moreover, conjugation invariance of trace gives
\[
 \Phi(g\cdot n)(t)=\eps(g)\Phi(n)(g^{-1}\cdot t)
                 =(g\cdot\Phi(n))(t),
\]
which proves $G_K$-equivariance.
After inverting $p$ and reducing values modulo $\mathcal O$, we obtain
\[
 \begin{aligned}
 \Psi:D=N[1/p]/N&\xrightarrow{\sim}
 \Hom_{\mathrm{cts},\mathcal O}(T,F/\mathcal O)(1)=D_T,\\
 p^{-r}n+N&\longmapsto
 \bigl[t\longmapsto p^{-r}\tr(t\circ n)+\mathcal O\bigr].
 \end{aligned}
\]
By the isomorphism (\ref{firstiso}), the $F$-valued functional associated with
an element of $N[1/p]$ takes values in $\mathcal O$ if and only if
that element lies in $N$. Thus $\Psi$ is well-defined and injective.
Since $T$ is finite free, every $\mathcal O$-linear map
$f:T\to F/\mathcal O$ lifts, by lifting its values on a basis,
to a map $T\to F$ with a common $p$-power denominator.
This proves surjectivity. It also shows that $f$ vanishes on
$p^rT$ for some $r$, and hence is continuous.
The $G_K$-equivariance of $\Psi$ follows from that of $\Phi$.

By Proposition~\ref{prop:integral-local-duality}, we have an isomorphism of $\mathcal{O}$-modules
\[
 H^2(G_K,T)\simeq\Hom_{\mathcal O}(D^{G_K},F/\mathcal O).
\]
If $D^{G_K}$ is $p$-divisible, then multiplication by $p$ on its dual is
injective. Thus it remains to prove that $D^{G_K}$ is $p$-divisible. It suffices to show that every $d\in D^{G_K}$ has a $G_K$-invariant representative in $N[1/p]$. 

To see this, we write $d=p^{-r}n+N$ with
$n\in N$ and $r\geq0$, and argue by induction on $r$.
The case $r=0$ is immediate. For $r\geq1$, invariance of $d$
gives $gn-n\in p^rN$ for every $g\in G_K$, and hence $n+pN\in(N/pN)^{G_K}$. Under the canonical identification
$B(1)/pB(1)\simeq(B/pB)(1)$, the hypothesis on \eqref{eq:hom-surj} says that
$N^{G_K}\to(N/pN)^{G_K}$ is surjective. Thus we can choose $m\in N^{G_K}$ with $n-m=py$
for some $y\in N$. Then
\[
 d-(p^{-r}m+N)=p^{-(r-1)}y+N\in D^{G_K}.
\]
By induction, there is $x\in N[1/p]^{G_K}$ such that
$x+N=p^{-(r-1)}y+N$. Since $m\in N^{G_K}$, the element
$x+p^{-r}m$ is $G_K$-invariant and satisfies
\[
 (x+p^{-r}m)+N=d.
\]
This completes our proof.
\end{proof}

\section{Proof of Theorem \ref{thm:main}}\label{sec:socle}

We continue to assume that $K$ is nonarchimedean and
$\charac K\ne p$. For a $k[G_K]$-module $\bV$ which is finite-dimensional over $k$, the \emph{socle} $\Soc(\bV)$ of $\bV$ is the sum of its irreducible submodules. Note that it is semisimple and is nonzero when $\bV\ne0$.

\begin{theorem}\label{thm:socle}
Let $\bV$ be a finite-dimensional $k$-vector space with a continuous
$k$-linear $G_K$-action for the discrete topology. Fix a decomposition
\[
 \Soc(\bV)=\bigoplus_{i=1}^t\rhobar_i^{m_i},
\]
where the $\rhobar_i$ are pairwise nonisomorphic irreducible representations of $G_{K}$.
Then there exist a finite free $W(k)$-module $V$ with a continuous
$W(k)$-linear $G_K$-action for the $p$-adic topology, lifting $\bV$,
and a $G_K$-stable sublattice $S\subset V$ such that:
\begin{enumerate}[label=\textup{(\roman*)},leftmargin=2.4em]
 \item the inclusion $S\hookrightarrow V$ reduces to
$\Soc(\bV)\hookrightarrow\bV$;
\item $S=\bigoplus_{i=1}^t M_i^{m_i}$, where each
      $M_i$ lifts $\rhobar_i$ and the map
      $\End_{G_K}(M_i)\to\End_{G_K}(\rhobar_i)$ is surjective.
\end{enumerate}
\end{theorem}

\begin{proof}
	We proceed by induction on the dimension $\dim_k\bV$ of $\bV$. First of all, the case $\bV=0$ is immediate. If
$\bV$ is semisimple, then we choose the $M_i$ using
Corollary~\ref{cor:local-simple} and take
$S=V=\bigoplus_{i=1}^t M_i^{m_i}$.

Suppose that $\bV$ is not semisimple. Set
\[
 \bS=\Soc(\bV),\qquad \bQ=\bV/\bS,\qquad
 \bB=\Soc(\bQ).
\]
Since $\bS\ne0$,  we have $\dim_k \bQ <\dim_k \bV$. By induction, we obtain a lift $Q$
with a sublattice $B\subset Q$ lifting the embedded
socle $\bB$, and decompositions
\[
 B=\bigoplus_{j=1}^u(M'_j)^{a_j},\qquad
 \bB=\bigoplus_{j=1}^u(\rhobar'_j)^{a_j},
\]
where the $\rhobar'_j$ are pairwise nonisomorphic irreducible modules
and every map $\End_{G_K}(M'_j)\to\End_{G_K}(\rhobar'_j)$ is surjective.

For each irreducible constituent $\rhobar_i$ of $\bS$, we choose $M_i$
as follows. If $\rhobar_i(-1)\simeq\rhobar'_j$ for some $j$, then we set
\[
 M_i=M'_j(1)
\]
and choose an identification of its reduction with $\rhobar_i$.
Otherwise we choose $M_i$ using Corollary~\ref{cor:local-simple}. By Corollary~\ref{cor:local-simple}, every $M_i$ has
surjective reduction on endomorphisms. Put
\[
 S=\bigoplus_{i=1}^t M_i^{m_i}.
\]
Then we have $S/pS\simeq\bS$. We claim that
\begin{equation}\label{eq:key}
 \Hom_{G_K}(S,Q(1))\longrightarrow
 \Hom_{G_K}(\bS,\bQ(1))
 \quad\text{is surjective}.
\end{equation}
Let $f\in \Hom_{G_K}(\bS,\bQ(1))$. For each irreducible
summand $U$ of $\bS$, the image $f(U)$ is either zero or an
irreducible subrepresentation of $\bQ(1)$. Hence
$f(\bS)\subseteq\Soc(\bQ(1))=\bB(1)$.  Write $f_B:\bS\to\bB(1)$ for the resulting map. Relative to the
chosen decompositions of $\bS$ and $\bB(1)$, the $(j,i)$-block
$f_{ji}:\rhobar_i^{m_i}\to\bigl(\rhobar'_j(1)\bigr)^{a_j}$
is zero if $\rhobar_i\not\simeq\rhobar'_j(1)$, since
there are no nonzero maps between nonisomorphic irreducible modules.

Suppose that $\rhobar_i\simeq\rhobar'_j(1)$. The chosen
reduction identifications for $M_i=M'_j(1)$ identify each source
summand with $\rhobar'_j(1)$. The block $f_{ji}$ is then an
$a_j\times m_i$ matrix $(u_{\alpha\beta})$ with entries in
$\End_{G_K}(\rhobar'_j(1))$. Since
\[
 \End_{G_K}(M'_j(1))\longrightarrow
 \End_{G_K}(\rhobar'_j(1))
\]
is surjective, we can choose a lift
$\widetilde u_{\alpha\beta}\in\End_{G_K}(M'_j(1))$
of each entry. Then the matrix $(\widetilde u_{\alpha\beta})$ defines
a $G_{K}$-equivariant map
\[
 F_{ji}:M_i^{m_i}=\bigl(M'_j(1)\bigr)^{m_i}
       \longrightarrow\bigl(M'_j(1)\bigr)^{a_j},
\]
whose reduction is $f_{ji}$ under the chosen identifications.

Taking these lifts on the matching blocks and zero on all other
blocks gives a $G_{K}$-equivariant map $F_B:S\to B(1)$ reducing to $f_B$.
By condition \textup{(i)} of the induction hypothesis, the inclusion
$B(1)\hookrightarrow Q(1)$ reduces to the inclusion
$\bB(1)\hookrightarrow\bQ(1)$. Hence the composite
$F:S\xrightarrow{F_B}B(1)\hookrightarrow Q(1)$ reduces to $f$.
This proves \eqref{eq:key}.

Applying Proposition~\ref{prop:extension} to the pair $(S,Q)$ and the
given exact sequence
\[
 0\longrightarrow\bS\longrightarrow\bV
 \longrightarrow\bQ\longrightarrow0,
\]
we obtain an exact sequence
\[
 0\longrightarrow S\longrightarrow V\longrightarrow Q
 \longrightarrow0
\]
of finite free $W(k)$-modules with continuous $W(k)$-linear
$G_K$-actions and $G_K$-equivariant maps, reducing to that sequence.
The resulting
inclusion $S\subset V$ satisfies \textup{(i)}, and the chosen
decomposition of $S$ satisfies \textup{(ii)}, completing the induction.
\end{proof}

Finally, we prove our main theorem as follows.

\begin{proof}[Proof of Theorem~\ref{thm:main}]
If $K$ is nonarchimedean and $\charac K\ne p$, then it follows from Theorem~\ref{thm:socle}. If $\charac K=p$, then the assertion follows from \cite[Theorem~1.7]{Boc03}. Finally, for $K=\mathbb R$ or $\mathbb C$, it is \cite[Lemma~2.3]{Boc03}.
\end{proof}

\section*{Acknowledgements}
The author is grateful to Prof. Dr. Gebhard B\"ockle for helpful
correspondence in December 2024 and for pointing out that, to his
knowledge, the problem of lifting to the original Witt ring $W(k)$
was still open.

\begingroup
\renewcommand{\addcontentsline}[3]{}

\section*{Declaration of AI use}
During the preparation of this paper, the author used ChatGPT
to search for relevant literature for the introduction, to polish
the English writing of the paper, and to check the proofs.
\par
\endgroup

\end{document}